\documentclass[11pt]{amsart}
\usepackage{graphicx}
\usepackage{amssymb}
\newcommand{\Z}{\mathbb{Z}}
\newcommand{\q}{\mathfrak{q}}
\newcommand{\g}{\gamma}
\newcommand{\G}{\Gamma}

\newcommand{\R}{\mathbb{R}}
\newcommand{\Q}{\mathcal{Q}}
\newcommand{\w}{\omega}

\newtheorem{mylem}{Lemma}

\newtheorem{mythm}{Theorem}
\newtheorem{mydef}{Definition}
\newtheorem{myprop}{Proposition}
\theoremstyle{definition}
\newtheorem*{myrmk}{Remark}
\usepackage{epstopdf}
\usepackage{amsmath}
\title{Almost Primes in Thin Orbits of Pythagorean Triangles}
\author{Max Ehrman}
\begin{document}
\maketitle
\begin{abstract}
Let $F=x^2+y^2-z^2$, $x_0 \in \Z^3$ primitive with $F(x_0)=0$, and $\G \leq SO_F(\Z)$ be a finitely generated thin subgroup. We consider the resulting thin orbits of Pythagorean triples $x_0 \cdot \G$ -  specifically which hypotenuses, areas, and products of all three coordinates arise. We produce infinitely many $R$-almost primes in these three cases whenever $\G$ has exponent $\delta_\G>\delta_0(R)$ for explicit $R$, $\delta_0$. 
\end{abstract}
\tableofcontents
\section{Introduction}
\subsection*{The Affine Sieve}
Consider a finitely generated group $\Gamma \leq GL(n, \Z)$, a base point $x_0 \in \Z^n$, and form the orbit $\mathcal{O} = x_0 \cdot \Gamma$. We call $\Gamma$ $thin$ if it is of infinite index in the $\Z$ points of its Zariski closure. Let $f \in \mathbb{Q}[x_1, \cdots, x_n]$ be integral on $\mathcal{O}$, and look at $f(\mathcal{O}) \subset \Z$. We say that $(\mathcal{O}, f)$ is $primitive$ if for every $q$, there exists $x \in \mathcal{O}$ such that $(f(x), q) = 1$, and further impose this condition on our $(\mathcal{O}, f)$. Let $\mathcal{P}_R$ denote the set of $R$-almost primes - that is, integers with at most $R$ prime factors. Now, the sets $\mathcal{O}(f, R) = \{ x \in \mathcal{O} : f(x) \in \mathcal{P}_R\}$ increase as $R \to \infty$, and so we wonder if there exists some finite $R$ such that $\mathcal{O}(f, R)$ is Zariski dense in the orbit $\mathcal{O}$. Precisely, we seek an $R$  such that $Zcl(\mathcal{O}(f, R)) = Zcl(\mathcal{O})$, where $Zcl$ denotes the Zariski closure. If this happens, we denote the least such $R$ by $R_0(\mathcal{O}, f)$, and call it the $saturation$ $number$. Bourgain, Gamburd, and Sarnak originally studied this problem in \cite{BGS10}, proving that for forms of $GL_2$ in the above context $R_0(\mathcal{O}, f)$ always exists. Salehi Golsefidy and Sarnak later settled this for $GL_n$ and more generally in \cite{SS13}. Their method was to order by word length and use combinatorics to produce a spectral gap for $\G$. While effective, this mechanism in some instances produces estimates on $R_0$ that are weaker than what can be done using Archimedean theory. In \cite{S08}, Sarnak first cited the hypotenuses and areas of thin Pythagorean triangles as particularly interesting examples to optimize $R_0$ for, and we improve the best known bound for the area here.
\\ \indent Let $F(x, y, z) = x^2 + y^2 - z^2$. We consider the cone $F=0$ of all real Pythagorean triples, and the action of the stabilizer $SO_F(\R)$ on these triples. We call an integral triple $(x, y, z)$ $primitive$ if $(x, y, z)=1$. There is an ancient parametrization of these triples dating back to the Babylonians, given by 
\begin{align*}
(c, d) \Longleftrightarrow (x, y, z) = (d^2-c^2, 2cd, c^2 + d^2) \\
\end{align*}
Let $SO_F^\circ (\R)$ be the connected component of the identity in $SO_F(\R)$. This parametrization is equivalent to the spin double cover $\iota: SL(2, \R) \to SO^{\circ}_F(\R)$ given by 
\begin{align*}
\begin{pmatrix} a & b \\ c& d
\end{pmatrix}
\mapsto \begin{pmatrix} \frac{1}{2}(a^2 - b^2 - c^2 + d^2) & cd - ab & \frac{1}{2} (-a^2 - b^2 + c^2 + d^2) \\
bd-ac & bc+ad & ac+bd  \\
\frac{1}{2}(-a^2 + b^2 - c^2 + d^2) & ab+cd & \frac{1}{2} (a^2 + b^2 + c^2 + d^2) \\
\end{pmatrix}
\end{align*}
In fact, this double cover restricts to an isomorphism  $PSL(2, \R) \cong SO^\circ_F(\R)$. Now let $X_0 = \begin{pmatrix} 1 & 0 & 1 \end{pmatrix}$, and let $x_0 = \begin{pmatrix} 0 & 1 \end{pmatrix}$. We have $X_0 \cdot \iota (\g) = (x(x_0 \cdot \g), y(x_0 \cdot \g), z(x_0 \cdot \g))$.
We will consider the following three functions from $(c, d) \in \Z^2 \to \Z$:
 \begin{align*}
 \text{the hypotenuse} & = z  = c^2 + d^2\\
 \text{the area}& = \frac{1}{12} xy = \frac{1}{6}cd(d^2-c^2) \\
 \text{the product of coordinates} &= \frac{1}{60}xyz =\frac{1}{30}cd(d^2-c^2)(d^2+c^2)  \\
 \end{align*}
 The factors of $12$ and $60$ are there so that our orbits will be primitive - one can easily check that they will still be integral. 
Now, setting up our affine sieve type orbit, let $\G \leq SL(2, \Z)$ be finitely generated and thin, let $f$ be one of our three functions, and let $\mathcal{O} = x_0 \cdot \G$.  
\begin{mydef} For a discrete group $\Gamma \leq SL(2, \Z)$, the $critical$ $exponent$ $\delta_\Gamma$ is the abscissa of convergence of the Poincare series $L_{\G}(s) = \displaystyle\sum_{\g \in \G} \frac{1}{||\g(z)||^{2s}}$ for any norm $|| \cdot ||$ on $M_{3 \times 3}(\R)$.
\end{mydef}
 Patterson and Sullivan proved that it is equivalently the Hausdorff dimension of the limit set $\Lambda_\Gamma \subset \partial \mathbb{H}$ \cite{P76} \cite{S84}. This $\delta_\G$ is a measure of how `thin' our groups are - closer to 1 is less thin, and $\G$ is a lattice if and only if $\delta_\G =1$. We will assume $\Gamma$ is finitely generated, thin, and has no parabolic elements throughout. This last assumption is due to the fact that Kontorovich showed in \cite{K09} that if $\G$ has parabolic elements, classical methods yield better results.
Kontorovich \cite{K09} and Kontorovich-Oh \cite{KO12} first analyzed this problem,  finding
\begin{mythm}[Kontorovich \cite{K09}, Kontorovich-Oh \cite{KO12}] Fix notation as above. Then for $\G$ with $\delta>\delta_0$, 
\begin{align*}
R_0(\mathcal{O}, f) &\leq \begin{cases} 13 \mbox{ if } f=z & \delta_0 = .9992 \\
40 \mbox{ if } f=\frac{1}{12}xy & \delta_0 = .9958 \\
58 \mbox{ if } f=\frac{1}{60}xyz &  \delta_0= .9974\\
\end{cases}
\end{align*}
\end{mythm}
Later, Bourgain and Kontorovich improved what is known about thin hypotenuses, going `beyond expansion' by utilizing the dispersion method and bilinear forms in order to prove the following: 
\begin{mythm}[Bourgain-Kontorovich \cite{BK15}]
If $\Gamma$ satisfies $\delta_\Gamma>\delta_0=1-10^{-17}$, then $R_0(\mathcal{O}, f) \leq 4$ for $f=z$. 
\end{mythm}
Finally, Hong and Kontorovich discovered a larger stabilizing subgroup to decompose the sum along and executed this in all three cases, producing improvements in the remaining two cases:
\begin{mythm}[Hong-Kontorovich \cite{HK15}]
If $\Gamma$ satisfies $\delta_\G>\delta_0$, we have 
\begin{align*} R_0(\mathcal{O}, f) &\leq \begin{cases} 25 \mbox{ if } f=\frac{1}{12}xy & \delta_0 = .99994\\
37 \mbox{ if } f=\frac{1}{60}xyz & \delta_0 = .99513\\
\end{cases}
\end{align*}
\end{mythm}
In this paper, we primarily build on \cite{BK15}. We take their method and improve it by integrating the ideas of \cite{HK15} and more importantly by further analyzing the bilinear forms that arise. Additionally, we adapt it to be able to handle the cases of the area and the product of coordinates, relying on a key new observation that these problems can be attacked similarly. Our main theorem is
\begin{mythm}
If $\Gamma$ has no parabolic elements and $\delta_\Gamma>\delta_0$, we have 
\begin{align*}
R_0(\mathcal{O}, f) \leq \begin{cases} 4  \mbox{} \text{ if   f=z}  & \delta_0 = .984 \\
18 \mbox{} \text{ if   f=$\frac{1}{12}$xy} & \delta_0 =.9955  \\
26  \mbox{} \text{   f=$\frac{1}{60}$xyz} & \delta_0 =.9963  \\
\end{cases}
\end{align*}
\end{mythm}
\begin{myrmk}
The value of $R_0$ is not improved for $f=z$ as compared to Theorem 2, but the allowable value of $\delta_0$ has. For the area and product of coordinates, the values of $R_0$ have improved from Theorem 3 along with their respective values of $\delta_0$.
\end{myrmk}

\subsection*{Notation}
We will use the following standard notation throughout. Let $e_t(x)=e^{2\pi i x/t}$. We denote the cardinality of the set $X$ by $|X|$. The letter $p$ will always be a prime. The congruence class $a \mod q$ will also be denoted by $a (q)$ interchangeably. The indicator function on the set $A$ is denoted by $1_A$. We write $f \ll g$ if $\exists C : f \leq C g$. If $g$ depends on epsilon, $f \ll_\epsilon g$ means $\exists C(\epsilon) : f\leq C(\epsilon) g_\epsilon$. We will often simply write $\ll$ for $\ll_\epsilon$. It is understood that the particular constant C is allowed to change line to line.  The greatest common divisor of $m$ and $n$ is written $(m, n)$ and their least common multiple is $[m, n]$. The transpose of a matrix $\g$ is written $\g^t$. 
\subsection*{Acknowledgements}
The author is extremely grateful to Alex Kontorovich for his generous tutelage.
\section{Background}
\subsection*{Strong Approximation}
\begin{myprop}[Strong Approximation, \cite{MVW84}]
Let $\G \leq SL(2, \Z)$ be Zariski dense, and let $\pi_q$ be the projection $SL(2, \Z) \to SL(2, \Z/q\Z)$. Then there exists a number $\mathcal{B}$ with the following property: whenever $q$ is squarefree and of the form $q=dp_1 \cdots p_k$ with $d|\mathcal{B}$ and $(p_i, \mathcal{B})=1$, one has 
\begin{align*}
\pi_q(\G) \cong \pi_{d}(\G) \times SL(2, \Z/p_1\Z) \times \cdots \times SL(2, \Z/p_k\Z) \\
\end{align*}
\end{myprop}
We call this $\mathcal{B} = \mathcal{B}(\G)$ the $bad$ $modulus$ and will use this symbol throughout to denote this, making it clear from context which group it is derived from. The key property fundamental our sieve is that if $q$ is squarefree with $(q, \mathcal{B}) = 1$, then $\pi: \G \to SL(2, \Z/q\Z)$ is surjective. If we let $S$ be a finite set of generators for $\G$, this tells us that the Cayley graphs Cay($\pi_q(\G), \pi_q(S)$) are connected for $q$ squarefree, $(q, \mathcal{B}) =1$. Super approximation tells us that these in fact expand.

\subsection*{Super Approximation}
We consider the action of a subgroup $\Gamma \leq SL(2, \mathbb{Z})$ on $\mathbb{H}$ by linear fractional transformations. The Laplace-Beltrami operator $\Delta = y^2(\frac{\partial^2}{\partial^2 x} + \frac{\partial^2}{\partial^2y})$ acting on $L^2(\G \backslash \mathbb{H})$ then has spectrum $Spec(\G \backslash \mathbb{H})$. Lax-Phillips \cite{LP82} proved that this decomposes into a purely continuous part above $1/4$ and a finite discrete component on (0, 1/4), say  
\begin{center}
$0 < \lambda_0 \leq \lambda_1 \leq \cdots \leq \lambda_{max} < 1/4$
\end{center}
Patterson and Sullivan (\cite{P76}, \cite{S84}) proved that $\lambda_0 = \delta(1-\delta)$, where $\delta$ is the critical exponent for $\Gamma$. Let $\Gamma(q)$ be the kernel of $\pi_q: \G \to SL(2, \Z/q\Z)$ (the principal congruence group mod $q$), and now any eigenfunction $f \in L^2(\G \backslash \mathbb{H})$ lifts to an eigenfunction in $L^2(\G(q) \backslash \mathbb{H})$, giving us the reverse inclusion $Spec(\G(q)\backslash \mathbb{H}) \supset Spec(\G \backslash \mathbb{H})$. Let $Spec(\G(q) \backslash\mathbb{H})$ be
\begin{center}
$0 < \lambda_0(q) \leq \lambda_1(q) \leq \cdots \leq \lambda_{max}(q) < 1/4$
\end{center}
 We call the new elements $Spec(\G(q) \backslash \mathbb{H})^{new}$. We always have $\lambda_0(q)=\lambda_0$ and $\lambda_1(q)>\lambda_0$, but a priori for different $q$ these $\lambda_1(q)$ may approach $\lambda_0$. If there exists $\theta \in (0,\delta)$ such that $Spec(\G(q) \backslash \mathbb{H})^{new} \subset [\theta(1-\theta), 1/4)$ for all $(q, \mathcal{B})=1$, then we call $(\delta(1-\delta), \theta(1-\theta))$ the $spectral$ $gap$. The key result of Bourgain-Gamburd-Sarnak (\cite{BGS10} \cite{BGS11}) we will use throughout is that if $\delta> 1/2$, there is always some spectral gap, and moreover, by Gamburd \cite{G02} if $\delta>5/6$, we may take $\theta = 5/6$.
\subsection*{Counting in Thin Groups} 

\begin{mythm}[Bourgain-Kontorovich \cite{BK15}]
For each $T>0$, there exists a smoothed indicator function $\Upsilon_T: \Gamma \to [0, 1]$ satisfying 
\begin{align*}
\Upsilon_T(\g) = \begin{cases} 1 &\mbox{ if } || \g|| < \frac{9}{10}T \\
0 &\mbox{ if } ||\g||> \frac{11}{10}T
\end{cases}
\end{align*}
Let $\G$ have exponent $\delta>1/2$ and spectral gap $\theta< \delta$. Then we have
\begin{align*}
\sum_{\g \in \G} \Upsilon_T(\g) \sim C \cdot T^{2\delta}
\end{align*} as $T \to \infty$. Moreover, for any $\gamma_0 \in \G$, any q squarefree with $(q, \mathcal{B})=1$, and any subgroup $\tilde\G(q)$ satisfying $\Gamma(q) \leq \tilde\Gamma(q) \leq \G$, we have 
\begin{align*}
\sum_{\g \in \tilde\G(q)} \Upsilon_T(\g\g_0) = \frac{1}{[\G: \tilde\G(q)]} \sum_{\g \in \G} \Upsilon_T(\g)+ O(T^{2\theta})
\end{align*} The implied constant does not depend on $\g_0$ or $q$.
\end{mythm}

\section{Sieving Preliminaries}
Let $\mathcal{A}=\{a(n)\}$ be a finite sequence of non-negative numbers supported on $[-N, N]$.  If we seek to prove that there are primes or almost primes in $\mathcal{A}$, sieve theory reduces the problem to estimating $|\mathcal{A}_\q| = \displaystyle\sum_{n \equiv 0 (\q)} a(n)$ for all $\q$ squarefree. We would like to prove that $|\mathcal{A}_\q| = \beta(\q)\chi + r(\q)$,  where $\beta(\q)$ is a multiplicative non-negative function (the `local density'), $\chi$ is the sum of the full sequence (the `mass'), and $ r(\q)$ is the error term. We call $\alpha$ an $exponent$ $of$ $distribution$ if for all $\epsilon>0$, we have
\begin{align*}
\sum_{\q < N^\alpha} |r(\q)| \ll_\epsilon \chi^{1-\epsilon} \\
\end{align*}
The linear sieve of Greaves \cite{G86} and the higher dimensional sieve of Diamond, Halberstam, and Richert (\cite{DHR96}, \cite{DH97}) allows us to convert this $\alpha$ into an upper bound on $R_0$, and so the goal is to raise $\alpha$ as high as possible.
\subsection*{Setting up the dispersion method} Throughout the paper we will use the Archimedean norm $|| \begin{pmatrix} a & b \\ c& d \end{pmatrix}||^2 = a^2 + b^2 + c^2 + d^2$. A key observation given later demonstrates that to study our three chosen forms, $z, \frac{xy}{12}, $ and $ \frac{xyz}{60}$, it is enough to study merely $x, y, $ and $z$. Therefore let $f$ be one of our quadratic forms of interest, $x, y,$ and $z$. The simplest mechanism of setting up this sieve is to let $a_T(n)=\displaystyle\sum_{\g \in \G \atop{||\g|| < T}} 1_{\{f(x_0 \cdot \g) =n\}}$ for any $T>0$. To create cancellation, we define two multiplicative functions, supported on squarefree numbers and defined as follows on primes:
\begin{align*}
& \rho(p) = \frac{2p-1}{p^2} \\
& \Xi(p ; n) = 1_{\{n \equiv 0 (p) \}} - \rho(p) \\
\end{align*}
Let $D_g=D$ be the degree of our polynomial $g \in \{z, \frac{1}{12}xy, \frac{1}{60}xyz\}$ in $(c, d)$ - so $D \in \{2, 4, 6\}$. Recalling that $N$ is the size of the largest element in our sequence, we have $N = \displaystyle\max_{\g \in \G \atop{||\g||<T}}||f(x_0 \cdot \gamma)|| \sim T^D$. Now, in analyzing $|\mathcal{A}_\q|$, we can rearrange things as follows:
\begin{align*}
|\mathcal{A}_\q| &= \sum_{n \equiv 0 (\q)} a_T(n) = \sum_{\g \in \G \atop {||\g||< T}} 1_{\{f(x_0 \cdot \g) \equiv 0 (\q)\}}\\
& = \sum_{\g \in \G \atop {||\g||< T}} \prod_{p| \q} 1_{\{f(x_0 \cdot \g)\equiv 0 (p)\}} = \sum_{\g \in \G \atop {||\g||< T}} \prod_{p|q} (\Xi(p; f(x_0 \cdot \g)) + \rho(p)) \\
&=\sum_{q | \q}\rho(\frac{\q}{q}) \sum_{\g \in \G \atop {||\g||< T}}  \Xi(q; f(x_0 \cdot \g))
\end{align*}
Introducing this oscillation is an adaptation of Linnik's dispersion method taken from \cite{BK15}. In this paper, Bourgain and Kontorovich decompose the sum along the subgroups $\G_{x_0}(q) = \{ \g \in \G : x_0 \equiv x_0 \cdot \g \mod q \}$ as follows, making use of Theorem 5: 
\begin{align}
|\mathcal{A}_\q| &=\displaystyle\sum_{q | \q} \rho(\frac{\q}{q}) \sum_{\g \in \tilde\Gamma(q)} \sum_{\g_0 \in \G_{x_0}(q) \backslash \G} \Xi(q; f(x_0 \cdot \g \g_0)) 1_{||\g\g_0||<T} \\
& =\displaystyle\sum_{q | \q} \rho(\frac{\q}{q})\sum_{\g_0 \in \G_{x_0}(q) \backslash \G} \Xi(q; f(x_0 \cdot \g_0)) [\frac{C\cdot T^{2\delta} }{[\G: \G_{x_0}(q)]} + O(T^{2\theta})] = \mathcal{M}_\q + r(\q)
\end{align}
We thus see that lowering the index of the subgroup we decompose along is critical to lowering the bound on the error term coming from the spectral gap. 
In \cite{HK15}, Hong and Kontorovich achieve this by realizing that rather than requiring $x_0 \mod q$ to be preserved by the right action of $\gamma$, 
all they need is $f(x_0 \cdot \g_0) \equiv 0 \mod q \iff f(x_0 \cdot \g \g_0) \equiv 0 \mod q$. Accordingly, for any $x \in \Z^2$, let  
\begin{align*}
\Gamma_{<x>}(q) = \{\g \in \G: x \cdot \g \equiv a x \mod q, a \in (\Z/q\Z)^\times \}
\end{align*}
Now, rewriting (1), (2) with $\G_{<x_0>}(q)$ replacing $\G_{x_0}(q)$ greatly reduces the error terms $r(\q)$, and enables us to achieve a level of distribution $\alpha$ almost as large as $\frac{\delta - \theta}{D}$, where $D$ is the degree of the polynomial $f$ (in the hypotenuses case $D=2$). Hong and Kontorovich proved the best known saturation numbers for the area and the product of all three coordinates as in Theorem 2 using this method, but the machine of Bourgain-Kontorovich from \cite{BK15} produces a far stronger level of distribution in the hypotenuses case. As current technology achieves at best $\delta-\theta= \frac{1}{6}-\epsilon$, for the hypotenuse case of $f=z$ this gives $\alpha = \frac{1}{12}-\epsilon$, while Bourgain and Kontorovich obtain any $\alpha= \frac{7}{24}-\epsilon$ for sufficiently large $\delta$ \cite{BK15}. This results in an improvement of $R_0$ from 7 to 4. 
\\ \indent To achieve this improvement, Bourgain and Kontorovich redefine $a_T(n)$ to introduce another variable to play with. Specifically, let $x+y=1$, and let $Y = T^y$, $X = T^x$. Let $\Omega_Y = B_Y \cap \Gamma = \{ \g \in \G : ||\g|| < Y\}$, and now define 
\begin{align*} a_T(n) =\displaystyle \sum_{\g \in \G}\sum_{\w \in \Omega_Y} \Upsilon_X(\g) 1_{\{{f(x_0 \cdot \g \w)=n}\}}\end{align*} where $\Upsilon_X$ comes from Theorem 5. We are now summing a smoothed version of the multi-set $\{ \g \w: ||\g||<X, ||\w||<Y\}$ rather than a smoothed ball of size $T$. While these are quite similar, the extra variable we have introduced will prove crucial in bounding the error term. Following Bourgain and Kontorovich, we will cut the sum on divisors at some height $Q_0$, a small power of T, and use different estimates in the two regimes. Now, combining these we obtain  
\begin{align*}
 |\mathcal{A}_\q| & = \sum_{n \equiv 0 (\q)} a_T(n) = \sum_{n} a_T(n)1_{\{n\equiv 0 (q)\}} = \sum_n a_T(n) \prod_{p| \q} (\Xi(p; n) + \rho(p))\\ &=\displaystyle\sum_{q| \q} \sum_{n} a_T(n) \Xi(q; n) \rho(\frac{\q}{q}) = \sum_{q| \q} \rho(\frac{\q}{q}) \sum_{\gamma \in \Gamma} \Upsilon_X(\gamma)\sum_{\omega \in \Omega_Y}\Xi(q; f(x_0 \cdot \gamma\omega))  \\ &=  \sum_{q| \q \atop {q<Q_0}} \rho(\frac{\q}{q}) \sum_{\gamma \in \Gamma} \Upsilon_X(\gamma)\sum_{\omega \in \Omega_Y}\Xi(q; f(x_0 \cdot \gamma\omega)) +  \sum_{q| \q \atop{ q\geq Q_0}} \rho(\frac{\q}{q}) \sum_{\gamma \in \Gamma} \Upsilon_X(\gamma)\sum_{\omega \in \Omega_Y}\Xi(q; f(x_0 \cdot \gamma\omega))\\
&= M_\q + r(\q)
 \end{align*}
 The first term $M_\q$ is our main term.  Let $\text{sgn}(r(\q)) = \zeta(\q)$. To bound 
 \begin{align*}\displaystyle\mathcal{E} = \sum_{\q<Q} |r(\q)| = \sum_{\q<Q} \zeta(\q) \sum_{q| \q \atop{q>Q_0}}\rho(\frac{\q}{q})\sum_{\g \in \G} \Upsilon_X(\g)\sum_{\w \in \Omega_Y} \Xi(q; f(x_0 \cdot \g\w)) \\
 \end{align*} apply Cauchy-Schwartz in the $\g$ variable and change the order of summation. This gives: 
 \begin{align*}
 \mathcal{E}^2 &\ll X^{2\delta} \sum_{Q_0 \leq q, q' \leq Q} \sum_{\q \equiv 0 (q_1), \q<Q \atop{\q' \equiv 0 (q_1'), \q'<Q} } \zeta(\q)\zeta(\q')\rho(\frac{\q}{q})\rho(\frac{\q'}{q'})\sum_{\g, \g' \in \G} \sum_{\w, \w' \in \Omega_Y}\Xi(q; f(x_0 \cdot \g\w))\Xi(q'; x_0 \cdot \g'\w') \\
 \end{align*}
This sum turns out to be tractable by using Poisson summation and investigating quadratic forms. 

\section{Main Term Analysis}
Throughout we take $q$ to be squarefree, and always assume that $2|\mathcal{B}$ (it does not harm us to throw more primes into $\mathcal{B}$). Let $\eta(q)$ be a multiplicative function defined on primes by $\eta(p) = p+1$, and let $q=p_1 \cdots p_l$ be its prime factorization. We begin with a lemma:
\begin{mylem}
$\G_{<x_0>}(q) \backslash \Gamma \cong \prod_{p|q}(\G_{<x_0>}(p) \backslash \Gamma)$. Additionally, $\{(0, 1)\} \cup \{(1, d): d \in \Z/p\Z\}$ form a complete set of coset representatives for $\G_{<x_0>}(p) \backslash \Gamma$. Therefore, the index $[\G_{<x_0>}(q): \Gamma] = \eta(q)$. 
\end{mylem}
\begin{proof}
$\Gamma$ surjects onto $SL(2, \Z/q\Z)$, and as $\G_{<x_0>}(q)$ contains the kernel of this map, $\G(q)$, we have $\G_{<x_0>}(q) \backslash \G \cong \G_{<x_0>}(q) \backslash SL(2, \Z/q\Z)$. By the Chinese Remainder Theorem for $SL(2, \Z/q\Z)$ we have $SL(2, \Z/q\Z) \cong \prod_{p|q} SL(2, \Z/p\Z)$, where the isomorphism is simply projection to each coordinate. There is a natural inclusion map $ \iota: \G_{<x_0>}(q) \to \prod_{p|q} \G_{<x_0>}(p)$. Let $\gamma \in i^{-1}(\g_1, \cdots, \g_l)$ with $\g_k \in \G_{<x_0>}(p_k)$ $\forall$ $1\leq k \leq l$, so for each $p_k| q$, we have $x_0 \cdot \gamma \equiv x_0 \cdot \gamma_k \equiv a_k x_0 \mod p_k$ for some $a_k\in (\Z/p_k \Z)^\times$. Thus if $a$ is the image of $(a_1, \cdots, a_l)$ under the canonical isomorphism $\prod_{p|q}(\Z/p\Z)^\times \cong (\Z/q\Z)^\times$, we see that $x_0 \cdot \gamma  \equiv ax_0 \mod q$, and so in fact $\G_{<x_0>}(q) \cong \prod_{p|q} \G_{<x_0>}(p)$ under the inclusion map. Therefore our stated identity $\G_{<x_0>}(q) \backslash \G \cong \prod_{p|q}(\G_{<x_0>}(p) \backslash \G)$ follows directly from the Chinese Remainder Theorems for $\Z$ and $SL(2, \Z)$. 
\end{proof}
\begin{myrmk} These three polynomials are in fact equivalent on primes $p \equiv 1(4)$. Indeed, fix $\epsilon_p \in \Z$ with $\epsilon_p^2 \equiv -1 (p)$, and let $\nu_p =\begin{pmatrix}1 & 1 \\ \epsilon_p & -\epsilon_p \end{pmatrix}$. Let $\nu = \begin{pmatrix} 1 & 1 \\ -1 & 1 \end{pmatrix}$. Then observe that $x( x_0 \cdot \g\nu) = y(x_0 \cdot \g)$ and $z(x_0 \cdot \g\nu_p) \equiv z(x_0 \cdot \g) \mod p$. We have $x_0 \cdot \G = (x_0 \cdot \nu^{-1})\cdot (\nu \G \nu^{-1})\nu = b \cdot \G' \nu$ for a conjugate group $\Gamma'$ which satisfies all of the same properties that $\G$ does. These lie in a copy of $SL_2$ rather than $SL_2(\Z)$ within $SL_2(\mathbb{Q})$, but all of our arguments hold in this broader context. Doing this as well for $\nu_p$ in place of $\nu$, and observing that as our counts are local varying the conjugate $\G'$ with $p$ doesn't matter, we see that estimating $|\mathcal{A}_\q|$ for any $\q$ squarefree with all prime factors $\equiv 1(4)$ is equivalent for any of $x, y, z$. Despite our awareness of this, we will do all of the computations directly in all three cases as it's not too difficult. 
\end{myrmk}

\begin{mythm} Let $\beta_f(p) = \begin{cases} \frac{2}{p+1} & \text{if } p\equiv 1(4), (p,\mathcal{B})=1, f=z \\ \frac{2}{p+1} & \text{ if } (p, \mathfrak{B})=1, f = x, y \\ 0 & \text{else} \end{cases}$ be a multiplicative function defined as such on primes, and let $\chi = |\mathcal{A}| = |\Omega_Y| \displaystyle\sum_{\gamma \in \Gamma} \Upsilon_X(\gamma)$. Then we can write $|\mathcal{M}_{\q}| = \beta(q) \chi + r^{(1)}(q) + r^{(2)}(q)$ where
\begin{align*}
\displaystyle\sum_{q<\mathcal{Q}} |r^{(1)}(q)| < \frac{\chi\mathcal{Q}^\epsilon Q_0 }{X^{2(\delta - \theta)}}
\end{align*}
and
\begin{align*}
\displaystyle\sum_{q<\mathcal{Q}} |r^{(2)}(q)| < \frac{\chi\mathcal{Q}^\epsilon}{Q_0}
\end{align*}
\end{mythm}

\begin{proof}

\begin{equation*}
\begin{split}
\mathcal{M}_\mathfrak{q} &= \displaystyle\sum_{q|\mathfrak{q} \atop q<Q_0} \rho(\frac{\mathfrak{q}}{q}) \sum_{\gamma \in \Gamma} \Upsilon_X(\gamma) \sum_{\omega \in \Omega_Y} \Xi (q ; f(x_0\cdot \gamma\omega)) \\ 
  &= \displaystyle\sum_{q|\mathfrak{q} \atop q<Q_0} \rho(\frac{\mathfrak{q}}{q}) \sum_{\gamma \in \Gamma_{<x_0>}(q) \backslash \Gamma} \sum_{\gamma_0 \in \Gamma_{<x_0>}(q)}\Upsilon_X(\gamma_0\gamma) \sum_{\omega \in \Omega_Y} \Xi (q ; f(x_0\cdot \gamma_0\gamma\omega))\\
  &= \displaystyle\sum_{q|\mathfrak{q} \atop q<Q_0} \rho(\frac{\mathfrak{q}}{q}) \sum_{\omega \in \Omega_Y} \sum_{\gamma \in \Gamma_{<x_0>}(q) \backslash \Gamma} \Xi (q ; f(x_0\cdot \gamma\omega)) \bigl[ \sum_{\gamma_0 \in \Gamma_{<x_0>}(q)}\Upsilon_X(\gamma_0\gamma)\bigr] \\
  &= |\Omega_Y|\displaystyle\sum_{q|\mathfrak{q} \atop q<Q_0} \rho(\frac{\mathfrak{q}}{q}) \sum_{\gamma \in \Gamma_{<x_0>}(q) \backslash \Gamma} \Xi (q ; f(x_0\cdot \gamma)) \bigl[ \sum_{\gamma_0 \in \Gamma_{<x_0>}(q)}\Upsilon_X(\gamma_0\gamma)\bigr] \\
  & = |\Omega_Y|\displaystyle\sum_{q|\mathfrak{q} \atop q<Q_0} \rho(\frac{\mathfrak{q}}{q}) \sum_{\gamma \in \Gamma_{<x_0>}(q) \backslash \Gamma} \Xi (q ; f(x_0\cdot \gamma)) \bigl[ \frac{CX^{2\delta}}{[\G : \G_{<x_0>}(q)]} + O(X^{2\theta}) \bigr] \\
  & = \mathcal{M}_\q^{(1)} + r^{(1)}(\q)
\end{split}
\end{equation*}

Here we have recognized that the $\omega$ may be taken out of the sum as we are ranging over the full quotient $\Gamma_{<x_0>}(q) \backslash \Gamma$, and used Theorem 2 to replace the sum inside the brackets.
\\ \indent Let $\omega(q)$ be the number of distinct prime divisors of $q$. We've shown above that 
\begin{align*}
\frac{1}{[\Gamma: \Gamma_{<x_0>}(p)]}\displaystyle\sum_{\gamma_0 \in \Gamma_{<x_0>}(p) \backslash \Gamma} 1_{\{f(x_0 \cdot \gamma_0) \equiv 0 (p)\}} = \beta(p)
\end{align*}
Using this together with $\rho(x) \ll x^\epsilon/x$ in analyzing $r^{(1)}(\q)$, we obtain
\begin{align*}
\bigl|r^{(1)}(\mathfrak{q})| &\ll |\Omega_Y|\displaystyle\sum_{q|\mathfrak{q} \atop q<Q_0} \rho(\frac{\mathfrak{q}}{q})  \sum_{\gamma \in \Gamma_{<x_0>}(q) \backslash \Gamma} \bigl |\Xi (q ; f(x_0\cdot \gamma))\bigr| X^{2\theta}\\ 
& \leq |\Omega_Y|X^{2\theta}\displaystyle\sum_{q|\mathfrak{q} \atop q<Q_o}\rho(\frac{\q}{q})\sum_{\gamma\in \Gamma_{<x_0>}(q) \backslash\Gamma}\prod_{p | q} (1_{\{f(x_0\cdot\gamma) \equiv 0 (p)\}} + \rho(p))\\ 
& = |\Omega_Y|X^{2\theta}\displaystyle\sum_{q|\q \atop q<Q_o}\rho(\frac{\q}{q})\sum_{\gamma \in \Gamma_{<x_0>}(q) \backslash \Gamma}\sum_{q' | q} 1_{\{f(x_0\cdot \gamma \equiv 0(q')\}}\rho(\frac{q}{q'})
\end{align*}
Now, we use Lemma 1 to evaluate
\begin{align*}
\displaystyle\sum_{\g \in \G_{<x_0>}(q)}\sum_{q'|q}1_{\{f(x_0 \cdot\gamma) \equiv 0 (q')\}}
\end{align*} In order to satisfy $f(x_0 \cdot \gamma) \equiv 0 (q')$, for any $p$ dividing $q'$, we need the coordinate of $\gamma$ to be one of exactly two in $\Gamma_{<x_0>}(p) \backslash \Gamma$ which do this. However, for all other $p |q$, it can be anything. Before proceeding we must first define the $divisor$ $function$ $d(n)$ and recall an elementary fact about it for any $\epsilon>0$:
\begin{align*}
d(n) \equiv \sum_{t|n} 1 \ll n^\epsilon
\end{align*}
Therefore in total we obtain  
\begin{align*}
\displaystyle\sum_{\g \in \G_{<x_0>}(q)}\sum_{q'|q}1_{\{f(x_0 \cdot\gamma) \equiv 0 (q')\}}=2^{\omega(q')}\gamma(\frac{q}{q'}) = d(q')\gamma(\frac{q}{q'})
\end{align*} We plug this in to complete the calculation. 
\begin{align*}
|r^{(1)}(q)| & \ll |\Omega_Y|X^{2\theta}\displaystyle\sum_{q|\q \atop q<Q_o}\rho(\frac{\q}{q})\sum_{\gamma \in \Gamma_{<x_0>}(q) \backslash \Gamma}\sum_{q' | q} 1_{\{f(x_0\cdot \gamma) \equiv 0(q')\}}\rho(\frac{q}{q'}) \\
& \ll |\Omega_Y|X^{2\theta}\displaystyle\sum_{q|\q \atop q<Q_o}\rho(\frac{\q}{q})\sum_{q' | q}d(q')\gamma(\frac{q}{q'})\rho(\frac{q}{q'})\\
& =|\Omega_Y|X^{2\theta}\displaystyle\sum_{q'|q|\q \atop q<Q_o}d(q')\gamma(\frac{q}{q'})\rho(\frac{\q}{q'})\\
& \ll |\Omega_Y|X^{2\theta} \sum_{q'|q|\q \atop {\q < Q_0}} q'^\epsilon (\frac{q}{q'}) (\frac{\q}{q'})^{\epsilon-1}\ll |\Omega_Y|X^{2\theta} \q^{\epsilon-1}\sum_{q'|q|\q \atop {\q < Q_0}} q \ll |\Omega_Y| X^{2 \theta} \frac{\q^\epsilon}{\q} Q_0 \\
\end{align*}
giving us
\begin{align*}
\ll \displaystyle\sum_{\mathfrak{q}<Q} |r^{(1)}(\q)| \ll |\Omega_Y|X^{2\theta}Q_o \displaystyle\sum_{\mathfrak{q}<Q}\q^\epsilon/\q \ll \chi Q_o Q^\epsilon/ X^{2(\delta-\theta)} 
\end{align*} as stated.

Now, as in \cite{BK15} we add and subtract the large $\q$ factors back into $\mathfrak{M}_{\q}^{(1)}$ to obtain the main term $\mathfrak{M}_{\q}^{(2)}$ and another error term. Precisely, we write
\begin{equation*}
\begin{split}
 \mathfrak{M}_{\q}^{(1)} &= \chi \displaystyle\sum_{q | \q \atop{\q<Q_0}} \rho(\frac{\q}{q}) \displaystyle\sum_{\gamma_0 \in \Gamma_{<x_0>}(q) \backslash \Gamma} \frac{\Xi(q; f(x_0\cdot \gamma_0))}{[\Gamma: \Gamma_{x_0}(q)]}\\
 & =  \chi \displaystyle\sum_{q | \q} \rho(\frac{\q}{q}) \displaystyle\sum_{\gamma_0 \in \Gamma_{<x_0>}(q) \backslash \Gamma} \frac{\Xi(q; f(x_0\cdot \gamma_0))}{[\Gamma: \Gamma_{x_0}(q)]} -  \chi \displaystyle\sum_{q | \q \atop{\q \geq Q_0}} \rho(\frac{\q}{q}) \displaystyle\sum_{\gamma_0 \in \Gamma_{<x_0>}(q) \backslash \Gamma} \frac{\Xi(q; f(x_0\cdot \gamma_0))}{[\Gamma: \Gamma_{x_0}(q)]} \\
 & = \mathfrak{M}_\q^{(2)} + r^{(2)}(\q)
\end{split}
\end{equation*}
Using Lemma 1, we can calculate precisely that $\mathfrak{M}_\q^{(2)}= \chi \beta(\q)$:
\begin{equation*}
\begin{split} \mathfrak{M}_{\q}^{(2)} &= \chi \displaystyle\sum_{q | \q} \rho(\frac{\q}{q}) \displaystyle\sum_{\gamma_0 \in \Gamma_{<x_0>}(q) \backslash \Gamma} \frac{\Xi(q; f(x_0\cdot \gamma_0))}{[\Gamma: \Gamma_{x_0}(q)]}\\ 
&=\chi\rho(\q)\sum_{q | \q} \prod_{p|q} \g(p)^{-1}\rho(p)^{-1}\sum_{\gamma_0 \in \G_{<x_0>}(p) \backslash \Gamma} \Xi(p; f(x_0 \cdot \gamma_0)) \\
& =\chi\rho(\q)\sum_{q | \q}  \prod_{p|q}\g(p)^{-1} \rho(p)^{-1}(2 - \g(p)\rho(p)) \\
& = \chi\rho(\q)\sum_{q | \q}\prod_{p|q} (\frac{2}{p+1}\frac{p^2}{2p-1} - 1) \\
& = \chi\rho(\q)\prod_{p | \q} (1 + (\frac{2}{p+1}\frac{p^2}{2p-1} - 1)) = \chi\beta(\q) \\
\end{split}
\end{equation*}
The final bit to wrap up is bounding $r_\q^{(2)}$. We have 
\begin{equation*}
\begin{split}
r_\q^{(2)} &= \chi \displaystyle\sum_{q | \q \atop{q \geq Q_0}} \rho(\frac{\q}{q}) \displaystyle\sum_{\gamma_0 \in \Gamma_{<x_0>}(q) \backslash \Gamma} \frac{\Xi(q; f(x_0\cdot \gamma_0))}{[\Gamma: \Gamma_{x_0}(q)]} \\
&= \chi\rho(\q)\sum_{q | \q \atop{q \geq Q_0}}\prod_{p|q} (\frac{2}{p+1}\frac{p^2}{2p-1} - 1) \\
& \ll \chi q^\epsilon\rho(\q) \sum_{q| \q \atop{q \geq Q_0}} \prod_{p| q} \frac{1}{p} \ll \chi \q^\epsilon \frac{1}{\q}\frac{1}{Q_0}\\
\end{split}
\end{equation*}
so we obtain the desired result $\displaystyle\sum_{\q<\mathcal{Q}}|r_\q^{(2)}| \ll \chi\mathcal{Q}^\epsilon \frac{1}{Q_0}$, completing the proof.
\end{proof}

\section{Estimating $\mathcal{E}^2$}
We let $\mathcal{E} = \displaystyle\sum_{\q<Q} |r(\q)|$, where $r(\q)= \displaystyle\sum_{q| \q \atop{q\geq Q_o}} \sum_n a_T(n) \rho(\frac{\q}{q}) \Xi(q; n)$. Our goal is to prove a bound for $\mathcal{E}$ by analyzing $\mathcal{E}^2$. We proceed as in \cite{BK15}, making improvements in a couple of places. Our final result is
\begin{mythm}
$\mathcal{E}^2 = \mathcal{E}^2_\leq + \mathcal{E}^2_>$ with
\begin{align*}
\mathcal{E}_\leq^2 \ll \mathcal{Q}^\epsilon\chi^2\bigl[ \frac{X^{2(1-\delta)}}{Q_0} + T^{2(x-\theta)}\bigr ]
\end{align*}   and 
\begin{align*}
\mathcal{E}_>^2 \ll X^{2\delta} Q^\epsilon |\Omega_Y|^2 Y^4 Q^{8}/X^{3} \end{align*}
\end{mythm}
To begin, we let $\zeta(\q) =\text{sgn } r(\q)$, and we reverse the order of summation to write
\begin{equation*}
\begin{split}
\mathcal{E} &= \displaystyle\sum_{\q<\mathcal{Q}}\zeta(\q)\sum_na_T(n)\sum_{q |\q \atop{q \geq Q_0}} \Xi(q; n)\rho(\frac{\q}{q}) \\
&=\sum_{Q_0 \leq q< \mathcal{Q}} \sum_n a_T(n) \Xi(q; n) \sum_{\q \equiv 0 (q) \atop{\q < \mathcal{Q}}} \rho(\frac{\q}{q})\zeta(\q) \\
&=\sum_{Q_0 \leq q< \mathcal{Q}} \sum_n a_T(n) \Xi(q; n) \zeta_1(q)\\
\end{split}
\end{equation*}
where $\zeta_1(q)=\displaystyle\sum_{q |\q \atop{q \geq Q_0}} \zeta(\q)\rho(\frac{\q}{q})$. A simple calculation using $\rho(x) \ll x^{\epsilon-1}$ shows that $\zeta_1(q)\ll T^\epsilon$ for all $q<\mathcal{Q}$, assuming $\mathcal{Q}$ is a power of $T$ (which we will take it to be later). Now we put in the definition of $a_T(n)$ and apply Cauchy-Schwartz in the $\g$-variable, utilizing the fact that the bottom row of an element $\g \in \G$ is unique because $\G$ has no parabolic elements. We also replace $\Upsilon$ with a smooth sum $\Phi$ on $(c, d) \in \Z^2$ satisfying $\Phi(x) \geq 1$ for $x \in [-1, 1]$ and having Fourier transform $\hat\Phi$ supported in $[-1, 1]$.  In keeping with the notation from \cite{BK15}, we let $f((c, d) \cdot \w) = f_\w (c, d)$. 
\begin{equation*}
\begin{split}
\mathcal{E} &=\sum_{Q_0 \leq q< \mathcal{Q}} \sum_{\g \in \G \atop{\g = \begin{pmatrix} * & * \\ c& d\end{pmatrix}}} \Upsilon_X(\g)\sum_{\omega \in \Omega_Y}\Xi(q; f(x_0 \cdot \g\w)) \zeta_1(q)\\
\mathcal{E}^2 &\leq X^{2\delta} \sum_{\g \in \G \atop{\g = \begin{pmatrix} * & * \\ c& d\end{pmatrix}}} \Phi(\frac{c}{X})\Phi(\frac{d}{X})\bigl| \sum_{Q_0 \leq q <\mathcal{Q}} \sum_{\w \in \Omega_Y} \Xi(q; f_\omega(c, d)) \zeta_1(q)\bigr|^2 \\ 
\end{split}
\end{equation*}

Now we complete the sum to include all possible bottom rows $(c,d)$ rather than just those arising from $\G$ in order to be able to use Poisson summation. We now introduce some notation which will be used throughout the section: $\bar q = [q, q']$, $\tilde{q} = (q, q')$, $q_1 = \frac{q}{\tilde{q}}$, $q_1' = \frac{q'}{\tilde{q}}$. Rearrange the sum and expand the square to split the sum along the q-variable and then apply Poisson summation:
\begin{align*}
&\mathcal{E}^2 \leq T^\epsilon X^{2\delta} \sum_{Q_0\leq q, q' < \mathcal{Q}} \bigl| \sum_{\w, \w' \in \Omega_Y} \sum_{c, d} \Phi(\frac{c}{X})\Phi(\frac{d}{X})\Xi(q; f_\w(c, d)) \Xi(q'; f_{\w'}(c,d))\bigr| \\
&= T^\epsilon X^{2\delta} \sum_{Q_0\leq q, q' < \mathcal{Q}} \bigl| \sum_{\w, \w' \in \Omega_Y} \sum_{c, d (\bar q)} \Xi(q; f_\w(c, d)) \Xi(q'; f_{\w'}(c,d)) \sum_{m, n \in \Z}\Phi(\frac{c + m\bar q}{X})\Phi(\frac{d+n \bar q}{X}) \bigr| \\
(\dagger)&=T^\epsilon X^{2\delta} \sum_{Q_0\leq q, q' < \mathcal{Q}} \bigl| \sum_{\w, \w' \in \Omega_Y} \sum_{c, d (\bar q)} \Xi(q; f_\w(c, d)) \Xi(q'; f_{\w'}(c,d)) \frac{X^2}{\bar q^2}\sum_{m, n \in \Z} \hat\Phi(\frac{mX}{\bar q})\hat\Phi(\frac{nX}{\bar q})e_{\bar q}(cm+dn)  \bigr| \\
\end{align*}
This last line shows how we can profit from splitting the sum again along $\bar q<X$ or $\bar q \geq X$. Accordingly, we write $\mathcal{E}^2 = \mathcal{E}^2_\leq+ \mathcal{E}^2_{>}$ to refer to the above sum on $\bar q \leq X$ and $\bar q > X$ respectively, and we begin with $\mathcal{E}^2_<$.

\subsection*{Bounding $\mathcal{E}^2_<$}
\begin{myprop}
$\mathcal{E}^2_< \ll Q^\epsilon \chi^2 ( \frac{X^{2(1-\delta)}}{Q_0}  + \frac{X^{2(1-\delta)}}{Y^{2(\delta-\theta)}})$
\end{myprop}
Simplifying $(\dagger)$, as $m=n=0$ is the only contributing term from the last sum, we have 
\begin{align*}
\mathcal{E}^2_< \leq T^\epsilon X^{2(\delta+1)} \sum_{Q_0 \leq q, q'< Q \atop {\bar q< X}} |\sum_{\w, \w' \in \Omega_Y} \sum_{c, d( \bar q)} \Xi(q; f_\w(c, d)) \Xi(q'; f_{\w'}(c, d))\frac{1}{\bar q^2} | 
\end{align*}
Now, we analyze these oscillating sums precisely. Let 
\begin{align*} S_1(q; \w) &=\frac{1}{q^2} \sum_{c, d (q)} \Xi(q; f_\w(c, d)) \\
S_2(q; \w, \w') &= \frac{1}{q^2} \sum_{c, d (q)} \Xi(q; f_\w(c, d))\Xi(q; f_{\w'}(c, d))
\end{align*}
\begin{mylem}
If $q>1,$ $ S_1(q; \w)=0$ for $f=f, g, h$. 
\end{mylem}
\begin{proof}
We have $S_1(q; \w) = \displaystyle\prod_{p|q}S_1(p; \w)$, so it is enough to prove it for primes. We have
\begin{align*}
p^2S_1(p; \w) = \sum_{c, d(p)} 1_{f_\w(c, d) \equiv 0 (p)} -(2p-1). 
\end{align*}
Above we are reduced to counting solutions to $f_\w(c, d) \equiv 0 (p)$ in the box $c, d (p)$. Noting that $c=d=0$ is always a solution, we count one and then realize that after the change of variables $\g \to \g \w$, this is equivalent to counting solutions to $f(c,d) = 0 (p)$ with $(c, d) \neq (0, 0)$. Now it is clear that for any of our possible $f$'s there are $2p-2$ solutions, plus one for $c=d=0$, giving us zero as requested.
\end{proof}

\begin{myprop} Let $\Omega_Y = \{ \gamma \in \Gamma : ||\gamma|| <Y\}$, $x \in \Z/q\Z \times \Z/q\Z$ with at least one coordinate coprime to q, and $f$ as above. If $f=z$, assume all primes $p|q$ are congruent to 1 mod 4. Then for $f=x,y,z$ we have 
\[ \displaystyle\sum_{\omega \in \Omega_Y} 1_{\{f(x \cdot\omega) \equiv 0(q)\}} \ll \frac{|\Omega_Y|}{q^{1-\epsilon}} + q^{\epsilon}Y^{2\theta}\]
\end{myprop}

\begin{proof}
We begin by decomposing the sum into
\begin{equation*}
\begin{split}
& \displaystyle\sum_{\omega \in \Omega_Y} 1_{\{f(x\cdot\omega) \equiv 0(q)\}} = \displaystyle\sum_{\omega_0 \in \G_{<x>}(q) \backslash \Gamma} \displaystyle\sum_{\gamma \in \G_{<x>}(q)}1_{\{f(x\cdot\gamma\omega_0) \equiv 0(q)\}}1_{\{||\gamma\omega_0||<Y\}}\\
& = \displaystyle\sum_{\omega_0 \in \G_{<x>}(q) \backslash \Gamma} 1_{\{f(x\cdot\omega_0) \equiv 0(q)\}} \displaystyle\sum_{\gamma \in \G_{<x>}(q)}1_{\{||\gamma\omega_0||<Y\}} \\ 
& \ll \displaystyle\sum_{\omega_0 \in \G_{<x>}(q) \backslash \Gamma}1_{\{f(x\cdot \omega_0)\equiv 0(q)\}} (\frac{1}{[\Gamma: \G_{<x>}(q)]}|\Omega_Y| + O(Y^{2\theta}))\\
\end{split}
\end{equation*} making use of Theorem 5 in the last line.
It remains to count precisely when $f(x\cdot \omega_0) \equiv 0(q)$.  We begin by noticing that by our choice of $x$, there always exists a $\eta \in \G$ such that $x \equiv \eta \cdot x_0 (q)$. Now, we can see that $\G_{<x>}(q) = \G_{<x_0>}(q)$ - indeed, let $\w \in \G_{<x_0>}(q)$. Then $x_0 \cdot w \equiv a (q) \Leftrightarrow \eta x_0 w \equiv a \eta x_0 = a x (q)$. Therefore, we have 
\begin{align*}
& \sum_{\omega_0 \in \Gamma_{<x>}(q) \backslash \Gamma} 1_{\{f(x \cdot \omega_0) \equiv 0(q)\}} & = \sum_{\omega_0 \in \Gamma_{<x>}(q) \backslash \Gamma} 1_{\{f(x_0\cdot \eta\omega_0)\equiv 0 (q)\}} &= \sum_{\omega_0 \in \Gamma_{<x>}(q) \backslash \Gamma}1_{\{f(x_0 \cdot \omega_0) \equiv 0 (q)\}} \\ 
\end{align*} (as we are summing over a full set of cosets). 

Now, utilizing Lemma 1, we have 
\begin{align*}
& \displaystyle\sum_{\omega_0 \in \G_{<x_0>}(q) \backslash \Gamma} 1_{\{f(x_0 \cdot \omega_0) \equiv 0(q)\}} = \sum_{\omega \in \prod_{p|q}\G_{<x_0>}(p) \backslash \Gamma} \prod_{p|q} 1_{\{f(x_0 \cdot \omega_p) \equiv 0(p)\}} \\  
&= \prod_{p|q} \sum_{\omega_0 \in \Gamma_{x_0}(p) \backslash \Gamma} 1_{\{f(x_0 \cdot \omega_0) \equiv 0(p)\}}
\end{align*}
Now, by Lemma 1 matrices with bottom rows $\{ (0, 1)\} \cup \{(1, d) | d\in \Z/p\Z \}$ form a complete set of coset representatives for $\G_{<x_0>}(p) \backslash \G$. We see that exactly two of these satisfy $f(x_0 \cdot \omega_0) \equiv 0 (p)$ for any $f$, $p$ that we allow. Therefore modulo p we obtain exactly 2 solutions, and so using Lemma 1, we see that if $q = p_1 \cdots p_l$, we have $2^l = d(q) \ll q^\epsilon$ solutions in total. Finally, inserting $[\G_{<x_0>}(p) : \G] = p+1$ and so $[\G_{<x_0>}(q) : \G] \asymp q$, we obtain the desired result. 
\end{proof}

We state the following lemma from \cite{BK15},  changing their $f$ (=$z$) to ours unimpeded. The proof comes easily from Proposition 3. 
\begin{mylem}
$|S_2|\ll \frac{q^\epsilon}{q^2} \displaystyle\sum_{q_1 |q} \sum_{c, d (q_1), \atop{(c, d, q_1) =1}} 1_{ f_\w(c, d)\equiv 0 (q_1) \atop{ f_\w'(c, d) \equiv 0(q_1)}}$
\end{mylem}
We have $ \frac{1}{\bar q^2} \displaystyle\sum_{c, d (\bar q)} \Xi(q; f_\w(c, d))\Xi(q'; f_{\w'}(c, d)) = S_1(q_1, \w)S_1(q_1', \w')S_2(\tilde q; \w, \w')$. In order for this not to be zero we need $q_1= q_1' = 1$, and so we have 
\begin{align*}
\mathcal{E}^2 &= \hat\Phi(0)^2 X^{2+2\delta} \sum_{Q_0 \leq q \leq Q}| \sum_{\w, \w' \in \Omega_Y} S_2(q; \w, \w')| \\
&\ll \hat\Phi(0)^2 X^{2+2\delta} \sum_{Q_0 \leq q \leq Q} \frac{q^\epsilon}{q^2}\sum_{t | q} \sum_{c, d (t)}\sum_{\w, \w' \in \Omega_Y}  1_{f_\w(c, d) \equiv 0(t) \atop{f_{\w'}(c, d) \equiv 0 (t)}}
\end{align*}
Now, fixing $\w$, on each prime there are $2p-1$ solutions to $f_\w(c, d) \equiv 0 (p)$, so in all we have $\ll t^{1+\epsilon}$ pairs $(c, d, t) =1$ such that $f_\w (c, d) \equiv 0 (t)$.
\begin{align*}
&\ll X^{2+2\delta} \sum_{Q_0 \leq q \leq Q} \frac{q^\epsilon}{q^2} \sum_{t|q} \sum_{\w \in \Omega_Y} \sum_{c, d (t) \atop{ f_\w (c,d)\equiv 0 (t)}} \sum_{\w' \in \Omega_Y} 1_{f_{\w'}(c, d) \equiv 0 (t)} \\
&\ll Q^\epsilon X^{2+2\delta} \sum_{Q_0 \leq q \leq Q} \frac{q^\epsilon}{q^2} \sum_{t|q} \sum_{\w \in \Omega_Y} \sum_{c, d (t) \atop{ f_\w (c,d)\equiv 0 (t)}}[  \frac{|\Omega_Y|}{t} + O(Y^{2\theta})]
\end{align*} 
where the last line is using Proposition 3. Now we break it into the corresponding two sums and see what we get:
\begin{align*}
&Q^\epsilon X^{2+2\delta} \sum_{Q_0 \leq q \leq Q} \frac{q^\epsilon}{q^2} \sum_{t|q} \sum_{\w \in \Omega_Y} \sum_{c, d (t) \atop{ f_\w (c,d)\equiv 0 (t)}}\frac{|\Omega_Y|}{t} \\
& \ll Q^\epsilon |\Omega_Y|^2X^{2+2\delta} \sum_{Q_0 \leq q \leq Q} \frac{q^\epsilon}{q^2} \sum_{t|q} \frac{t^{1+\epsilon}}{t} \ll \chi X^2|\Omega_Y|/Q_0 = \chi^2 \frac{X^{2(1-\delta)}}{Q_0} \\
\end{align*}
For the error term, we have 
\begin{align*}
&\ll Q^\epsilon X^{2+2\delta}Y^{2\theta}|\Omega_Y| \sum_{Q_0 \leq q \leq Q} \frac{1}{q^2}\sum_{t|q} t^{1+\epsilon} \\
&\ll Q^\epsilon X^{2+2\delta}Y^{2\theta}|\Omega_Y| \sum_{Q_0 \leq q \leq Q} \frac{1}{q^2}q\\
&\ll  Q^\epsilon X^{2+2\delta}Y^{2\theta}|\Omega_Y| = \chi^2 \frac{X^{2(1-\delta)}}{Y^{2(\delta-\theta)}} \\
\end{align*}
which we can see will require $Y$ or $\delta$ to be relatively large, a restraint we'd hope to drop.  

\subsection*{Bounding $\mathcal{E}^2_>$}
\begin{myprop}
\begin{align*}
\mathcal{E}^2_> \ll X^{2\delta}Q^\epsilon |\Omega_Y|^2 Y^4 Q^{8}/X^3
\end{align*}
\end{myprop}

 Now we rearrange the sum and expand the square to split the sum along the q-variable and then apply Poisson summation:
\begin{align*}
\mathcal{E}^2 &\leq T^\epsilon X^{2\delta} \sum_{Q_0\leq q, q' < \mathcal{Q}} \bigl| \sum_{\w, \w' \in \Omega_Y} \sum_{c, d} \Phi(\frac{c}{X})\Phi(\frac{d}{X})\Xi(q; Q_\w(c, d)) \Xi(q'; Q_{\w'}(c,d))\bigr| \\
&= T^\epsilon X^{2\delta}\sum_{Q_0 \leq q, q' \leq Q}\bigl| \sum_{\w, \w' \in \Omega_Y} \sum_{c, d (\bar q)} \Xi(q; f_\w(c, d)) \Xi(q'; f_{\w'}(c,d))\sum_{m, n \in \mathbb{Z}^2}\Phi(\frac{m}{X})\Phi(\frac{n}{X}) 1_{\{m\equiv c (\bar q)\}} 1_{\{n \equiv d (\bar q)\}}\bigr| \\
&= T^\epsilon X^{2\delta}\sum_{Q_0 \leq q, q' \leq Q}\bigl| \sum_{\w, \w' \in \Omega_Y} \sum_{c, d (\bar q)} \Xi(q; f_\w(c, d)) \Xi(q'; f_{\w'}(c,d))\sum_{m, n \in \mathbb{Z}^2}\Phi(\frac{m}{X})\Phi(\frac{n}{X}) \\ &\frac{1}{\bar q^2}\sum_{l (\bar q)} e_{\bar q} (l(m-c)) \sum_{k (\bar q)} e_{\bar q} (k(n-d))\bigr| \\
& = T^\epsilon X^{2\delta}\sum_{Q_0 \leq q, q' \leq Q}\bigl| \sum_{\w, \w' \in \Omega_Y} \frac{1}{\bar q^2}\sum_{c, d (\bar q)} \Xi(q; f_\w(c, d)) \Xi(q'; f_{\w'}(c,d)) e_{\bar q}(-cl-nk)\\\ &\sum_{m, n \in \mathbb{Z}^2}\Phi(\frac{m}{X})\Phi(\frac{n}{X})e_{\bar q}(ml + nk)\bigr| \\
& = T^\epsilon X^{2\delta} \sum_{Q_0 \leq q, q' \leq Q}\bigl|\sum_{\w, \w' \in \Omega_Y}\sum_{l, k (\bar q)}S_3(q, q'; l, k; \w, \w') \mathcal{I}(X; k, l; \bar {q})\bigr|
\end{align*}
where we have
\[ S_3(q, q'; k, l; \w, \w') =\frac{1}{\bar q^2} \sum_{c, d(q)} \Xi(q; f_\w(c, d))\Xi(q'; f_{\w'}(c, d)) e_{\bar q}(-ck-dl)\] and
 \[ \mathcal{I}(X; k, l; \bar q) = \sum_{m,n \in \mathbb{Z}} \Phi(\frac{m}{X})\Phi(\frac{n}{X})e_{\bar q}(ml + nk) \]. 
Now we use Poisson summation on $\mathcal{I}$, obtaining
\begin{align*}
 \mathcal{I}(X; k, l; \bar q) = \sum_{m,n \in \mathbb{Z}} \Phi(\frac{m}{X})\Phi(\frac{n}{X})e_{\bar q}(ml + nk) = (\sum_{m \in \Z} \Phi(\frac{m}{X})e_{\bar q}(ml)) (\sum_{n \in \Z} \Phi(\frac{n}{X})e_{\bar q}(nk))
\end{align*}
 To see this, compute the Fourier transform of the function $F(m) = \Phi( \frac{m}{X}) e_{\bar q}(ml)$: 
 \begin{align*}
 \hat F(m) &= \int_{-\infty}^\infty \Phi(\frac{t}{X})e_{\bar q}(tl)e(-tm)dt = \int_{-\infty}^\infty \Phi(\frac{t}{X})e(-t(m-l/\bar q))dt = X \int_{-\infty}^\infty \Phi(u)e(-Xu(m-l/\bar q))du \\
 &= X \hat\Phi(X(m - \frac{l}{\bar q})) 
 \end{align*}
 So, by Poisson summation we have 
 \begin{align*}
  \mathcal{I}(X; k, l; \bar q) &= X^2 \sum_{m, n \in \mathbb{Z}^2}\hat\Phi( X(m - \frac{l}{\bar q}))\hat\Phi( X(n - \frac{k}{\bar q})) \\
  & \ll X^2 1_{\{ l, k \leq \bar q/ X\}}
 \end{align*}
 
 Now, switching our focus to $S_3$, we introduce $S_4$ and $S_5$: 
 
\[ S_4(q; k, l; \w) = \frac{1}{q^2} \sum_{c, d(q)} \Xi(q; f_\w(c,d))e_q(-ck-dl) \]   
\[ S_5(q; k, l; \w, \w') = \frac{1}{q^2} \sum_{c, d(q)} \Xi(q; f_\w(c, d))\Xi(q; f_{\w'}(c, d)) e_q(-ck-dl)\]
Notice that $S_3(q, q'; k, l; \w, \w') = S_4(q; k, l; \w)S_4(q_1'; k, l; \w')S_5(\tilde q; k, l; \w, \w')$. Additionally, we will only use the trivial bound $|S_5|\leq 1$. We'll note that it is possible to save even more from the $S_5$ sum than in \cite{BK15}, but it proves unnecessary here. 
Now we break for a lemma concerning $S_4$.
\begin{mylem}
$|S_4(q; k, l; \w)| \leq \begin{cases}
\frac{f_\w(l, -k)}{q^2} &\mbox{ if } (k, l, q) =1, f=z \\
 \frac{(f_\w(k, l), q)}{q^2} &\mbox{if } (k, l, q) =1 , f=y \\
 \frac{(f_{\w^t}(k,l), q)}{q^2} &\mbox{if } (k, l, q) =1 , f=x \\
0 &\mbox {if } (k, l, q)>1
\end{cases}$
\end{mylem}
\begin{proof}
Noticing that $S_4(q; k, l; \w) = \prod_{p | q} S_4(p; k, l; \w)$, as always we reduce to analyzing $S_4$ at primes. If $(k, l, q)>1$, one of the local factors is $S_1(p; \w)$ which is zero. Suppose then that $(k, l, q)=1$. 
We have $p^2S(p; k, l, \w) = \displaystyle\sum_{c,d(p)}(1_{f_\w(c, d) \equiv 0 (p)} - \rho(p)) e_p(-ck-dl) = \sum_{c,d(p)}1_{f_\w(c, d) \equiv 0 (p)} e_p(-ck-dl)$
Now, let $\w = \begin{pmatrix} \alpha & \beta \\ \g & \delta \end{pmatrix}$, and we see that the above is equivalent to
\begin{align*}
\sum_{c, d(p)}1_{f((\alpha c + \gamma d)(\beta c + \delta d)) \equiv 0 (p)}e_p (-ck -dl) \\
\end{align*}
For simplicity we now subtract off the $c=d=0$ contribution of 1. Writing $\gamma= \begin{pmatrix} * & * \\ c & d \end{pmatrix} \in \G(p) \backslash \G$, we have $p^2S_4(p; k, l; w) - 1 = \displaystyle\sum_{\g \in \G(p)\backslash \G} 1_{f(x_0 \cdot \g \w) \equiv 0 (p)} e_p(-x_0 \cdot \g \cdot (k, l)^t)$. Changing $\g$ to $\g\w$, this becomes
\begin{align*}
p^2S_4(p; k, l; w) - 1 &= \sum_{\g \in \G(p)\backslash \G} 1_{f(x_0 \cdot \g) \equiv 0 (p)} e_p(-x_0 \cdot \g \w \cdot (k, l)^t) \\
&= \sum_{\g \in \G(p)\backslash \G} 1_{f(x_0 \cdot \g) \equiv 0 (p)} e_p(-(c, d) \cdot (\alpha k + \beta l, \gamma k + \delta l)^t)\\
& = \sum_{\g \in \G(p)\backslash \G} 1_{f(x_0 \cdot \g) \equiv 0 (p)} e_p(-c(\alpha k + \beta l)-d(\gamma k + \delta l))
\end{align*}
The claimed bound for $f$ was already proven in \cite{BK15}. We treat the two unsolved cases, $f=x$ and $f=y$, separately. Let $f=y$. Then in order to obtain a solution to $h(x_0 \cdot \g) = 0$, we need $c$ or $d$ to be zero. If this is the case, the other is not zero, so we may write
\begin{align*}
p^2S_4^y(p; k, l; \w) - 1&= \sum_{d \neq 0} e_p(-d(\gamma k + \delta l)) + \sum_{c \neq 0} e_p(-c(\alpha k + \beta l) \\
 S_4^y(p; k, l, \w) &= \begin{cases} \frac{p-1}{p^2} &\mbox{ if } \gamma k +\delta l \equiv 0 (p) \text{ or } \alpha k + \beta l \equiv 0(p) \\
 \frac{-1}{p^2} &\mbox{ else } \\
\end{cases}
\end{align*}
\\ Changing things around for notational simplicity, we see that this condition is equivalent to $y_{\w}(k, l) = y((k, l) \cdot \w) \equiv 0 (p) $, as stated.
\\ Now for the last case, $f =x=d^2-c^2$, we notice that the solutions are of the form $d = \pm c$ with $c \neq 0$. Expanding as above we obtain
\begin{align*}
p^2S_4^x(p; k, l; \w)-1 &= \sum_{c\neq 0 (p)} e_p(-c((\alpha + \gamma)k + (\beta + \delta)l)) + \sum_{c\neq 0 (p)} e_p(-c((\alpha - \gamma)k +(\beta - \delta)l)) \\
 S_4^x(p; k, l, \w) &= \begin{cases} \frac{p-1}{p^2} &\mbox{ if } (\alpha+\gamma)k + (\beta+\delta)l \text{ or } (\alpha-\gamma)k + (\beta-\delta)l \equiv 0(p) \\
 \frac{-1}{p^2} &\mbox{ else } \\
 \end{cases}
\end{align*}
Now, we notice that $x(c,d)=d^2-c^2 = (d+c)(d-c)$. Therefore the first condition is equivalent to $x((k, l) \cdot \w^t) = 0$ as claimed.
\end{proof}

Now, we finally treat $\mathcal{E}_>^2$. We will do this for the case of $f=z$- note that replacing $(l, -k)$ with $(k, l)$ makes no difference, nor does transposing the group $\Gamma$, so the cases of $x$ and $y$ follow in exactly the same way. Inserting estimates for $S_4$ and $S_5$ gives
\[ \frac{\mathcal{E}^2_<}{T^\epsilon X^{2\delta}} \ll Q^\epsilon X^2 \displaystyle \sum_{X<\bar q< Q^2} \sum_{q_1q_1'\tilde q=\bar q} \frac{ \tilde q^2 }{\bar q^2} \sum_{\omega, \omega' \in \Omega_Y} \sum_{0 \leq l, k \leq \frac{\bar q}{X}} (f_\omega(l, -k), q_1) (f_{\omega'}(l, -k), q_1') \] 

Our first step is to simply use the bound $(f_{\omega'}(l, -k), q_1') \leq f_{\omega'}(l, -k) \ll (\frac{\bar q}{X})^2Y^2$ which follows from the sizes of the matrix entries. Now, notice that for fixed $\omega$ and $l$, $f_\omega(l, -k)$ is merely a quadratic, let's say $f(k)= ak^2+bk+c$. For reasons that will soon be clear, we wish to assume that $(a, b, c, q_1) = 1$. As such, set gcd$(a, b, c, q_1)=d$, and let $f'(k) = f(k)/d$. Now, we have $(f(k), q_1) = d (f(k)/d, q_1/d) \leq d(f'(k), q_1)$. Moreover, $d \leq a=C \ll Y^2$ by the construction of $f$.  Therefore, what we have shown is that for fixed $\omega, l$, we may replace $f_\omega(l, -k)$ with a polynomial $p_{\omega, l}(k)$ with coprime coefficients at a cost of at most $O(Y^2)$. 

\ Now, if we take some fixed squarefree modulus $t = p_1 \cdots p_l$, we have $f(k) \equiv 0(t) \iff f(k)\equiv 0 (p_i) \forall i$, and as atleast one polynomial coefficient is a unit mod $p_i$, there are at most two solutions to this mod $p_i$. Using the Chinese Remainder Theorem, we see that we have at most $2^l$ solutions to $f(k) \equiv 0(t)$ modulo $t$, unless we fail this coprimality condition. As $t$ is squarefree, we have $2^l= d(t) \ll t^\epsilon$. Therefore, counting up to $\bar q/X$, we can break $\frac{\bar{q}}{X}$ into $\lceil \frac{\bar q}{Xt}\rceil \leq \frac{\bar q}{Xt}+1$ units of length t, obtaining a total of $\ll (\frac{\bar q}{Xt}+1)t^{1+\epsilon}$ solutions.  Now we return to the calculation, remembering that $\tilde q \bar q \leq Q^2$ and rewriting the sum to take advantage of the above observations: 

\begin{align*} &\mathcal{E_>}^2 \ll Q^\epsilon |\Omega_Y| Y^2 \displaystyle\sum_{X<\bar{q}<Q^2} \sum_{q_1q_1'\tilde{q} = \bar{q}}\tilde{q}^2 \sum_{\omega \in \Omega_Y}  \sum_{0 \leq k, l \leq \bar{q}/X \atop{(k, l) \neq (0, 0)}} (f_\omega(l, -k), q_1) \\
  &\ll Q^\epsilon |\Omega_Y| Y^2 \displaystyle\sum_{X<\bar{q}<Q^2} \sum_{q_1q_1'\tilde{q} = \bar{q}}\tilde{q}^2 \sum_{\omega \in \Omega_Y}  \sum_{0 \leq l \leq \bar{q}/X} \sum_{t|q_1} \sum_{0 \leq k \leq \bar q/X \atop{ f_\omega(l, -k) \equiv 0 (t)}} t  \ll \mathcal{E}_1 + \mathcal{E}_2\\ 
\end{align*}

We break the sum into two components, $\mathcal{E}_1, \mathcal{E}_2$ corresponding to $t\leq \bar{q}/X$ and $t> \frac{\bar{q}}{X}$ respectively. We use the fact that at most $t^{1+\epsilon}$ pairs $(l, k)$ mod $t$ can solve $f_\omega(l, -k) \equiv 0 (t)$. We treat them individually:
\begin{align*}
\mathcal{E}_1 &\ll Q^\epsilon |\Omega_Y| Y^2\displaystyle\sum_{X<\bar{q}<Q^2} \sum_{q_1q_1'\tilde q = \bar q}\tilde q^2\sum_{\omega \in \Omega_Y} \sum_{t| q_1 \atop{ t\leq\bar{q}/X}} \sum_{l, k \leq \bar{q}/X} t1_{f_\omega(l, -k) \equiv 0 (t)} \\
&\ll Q^\epsilon |\Omega_Y| Y^2\displaystyle\sum_{X<\bar{q}<Q^2} \sum_{q_1q_1'\tilde q = \bar q}\tilde q^2 \sum_{\omega \in \Omega_Y} \sum_{t| q_1 \atop{ t\leq\bar{q}/X}} t(\frac{\bar{q}^2/X^2}{t^2} t^{1+\epsilon}) \\
& \ll Q^\epsilon |\Omega_Y| Y^2X^{-2}\displaystyle\sum_{X<\bar{q}<Q^2} \sum_{q_1q_1'\tilde q = \bar q}\tilde q^2 \bar q^2 \\
&\ll Q^\epsilon|\Omega_Y|^2Y^2 Q^6/X^2
\end{align*}
The term which will in fact dominate is $\mathcal{E}_2$: 
\begin{align*}
\mathcal{E}_2 &\ll Q^\epsilon|\Omega_Y|Y^2 \displaystyle\sum_{X<\bar q< Q^2} \sum_{q_1q_1'\tilde{q} = \bar q} \tilde q^2\sum_{\omega \in \Omega_Y} \sum_{t | q_1 \atop { t \geq \bar q/X}}\sum_{l, k \leq \bar q/X} t 1_{f_\omega(l, -k) \equiv 0 (t)} \\
&\ll Q^\epsilon|\Omega_Y|^2Y^2 \displaystyle\sum_{X<\bar q< Q^2} \sum_{q_1q_1'\tilde{q} = \bar q} \tilde q^2 \sum_{t | q_1 \atop {\bar q/X \leq t \leq (\bar q/X)^2Y^2}} t \bar{q}/X \\
& \ll Q^\epsilon |\Omega_Y|^2 Y^4 \displaystyle \sum_{X<\bar q<Q^2} \sum_{q_1q_1'\tilde{q}=\bar{q}} \tilde q^2 \frac{\bar q^3}{X^3} \\
& \ll Q^\epsilon |\Omega_Y|^2 Y^4 X^{-3} Q^2 \displaystyle\sum_{\bar q<Q^2} (\bar q)^2 =Q^\epsilon |\Omega_Y|^2 Y^4 Q^{8}/X^3\\
\end{align*}
As $Q^2>X$, we have completed the proof of Theorem 7. 

\section{Conclusion}
\begin{mythm}
Fix $\epsilon>0$. There is an effective $\delta_0(\epsilon)<1$ such that for $\delta_\G > \delta_0(\epsilon)$, $f = x, y, z$, $\epsilon>0$, $\beta_f(p)=\beta(p)$ as in Theorem 5, the sequences $a_T(n)$ defined above satisfy 
\begin{align*}
|\mathcal{A}_\q| = \beta(\q)\chi + r(\q) \\
\end{align*}
with $r(\q)$ having a level of distribution of $\frac{5}{16}-\epsilon$.
\end{mythm}

We will now prove Theorem 8 and our main result Theorem 4 at once. First, we relay the simple trick that makes it possible to analyze areas and products using simply $f \in \{x, y, z\}$. The key observation is that no two of $x,y$ and $z$ are ever simultaneously zero modulo $q$ for any $q$ squarefree. Precisely, if $\gamma = \begin{pmatrix} a & b \\ c& d \end{pmatrix} \in \G \leq SL(2, \Z)$, it is impossible for more than one of $d^2 - c^2$, $cd$, and $c^2+ d^2$ to be congruent to 0 mod $q$. This is equivalent to no two of them being $ \equiv 0 (p)$ for any $p$ prime. Now suppose $c^2 + d^2, d^2 - c^2 \equiv 0 (p)$. Adding them together we obtain $2d^2 \equiv 0 (p)$, which, as we will assume $p \neq 2$, implies $d = 0$. Now $c^2 + d^2 \equiv 0 (p)$ implies $c$ is also 0, which is impossible as $\g \in SL(2, \Z)$. Now suppose $cd =0 (p)$. Without loss of generality, $c \equiv 0 (p)$. Now, regardless of if our other polynomial is $x$ or $z$, we have $d^2 \equiv 0 (p)$ and thus $d=0$, again a contradiction. Therefore we see that for any element $\g \in SL(2, \Z)$, any $q$ squarefree, we have at most one of $x(x_0 \cdot \g), y(x_0 \cdot \g), z(x_0 \cdot \g) \equiv 0 (q)$.
\\ We have thus shown that $1_{xyz \equiv 0 (q)} = 1_{x \equiv 0 (q)} + 1_{y \equiv 0 (q)} + 1_{z \equiv 0 (q)}$. This allows us to write:
\begin{align*}
|\mathcal{A}^{xyz}_\q| &= \sum_{\g \in \G} \sum_{\w \in \Omega_Y} \Upsilon_X(\g)1_{\{xyz(x_0 \cdot \g\w) \equiv 0 (q)\}} \\
& = \sum_{\g \in \G} \sum_{\w \in \Omega_Y} \Upsilon_X(\g) 1_{\{x(x_0 \cdot \g\w) \equiv 0 (q)\}} + \sum_{\g \in \G} \sum_{\w \in \Omega_Y} \Upsilon_X(\g)1_{\{y(x_0 \cdot \g\w) \equiv 0 (q)\}} + \\ &\sum_{\g \in \G} \sum_{\w \in \Omega_Y} \Upsilon_X(\g)1_{\{z(x_0 \cdot \g\w) \equiv 0 (q)\}} \\
&= |\mathcal{A}^x_\q| + |\mathcal{A}^y_\q| + |\mathcal{A}^z_\q| \\
\end{align*}
Similarly, $|\mathcal{A}^{xy}_\q| = |\mathcal{A}_\q|^x + |\mathcal{A}_\q|^y$. Recall that we denote the degree of our form $g \in \{z, \frac{1}{12}xy, \frac{1}{60}xyz\}$ in variables $(c, d)$ by $D$, so $D \in \{2, 4, 6\}$. 

Collecting error terms from Theorems 5-7, for $\mathcal{A}_q^f$ with $f \in \{x, y, z\}$ we have
\begin{align*}
&\sum_{q<Q} |r^{(1)}(\q)| \leq \chi \frac{ \Q^\epsilon \Q _0}{X^{2(\delta-\theta)}}\\
&\sum_{q< \Q} |r^{(2)}(\q)| \leq \chi \frac{\Q^\epsilon}{\Q _0} \\
&\mathcal{E}^2_< \leq \chi \Q^\epsilon (\frac{X^{1-\delta}}{Q_0} + \frac{X^{1-\delta}}{Y^{\delta-\theta}}) \\
&\mathcal{E}^2_> \leq \chi \Q^\epsilon \frac{Q^4Y^2}{X^{(3+2\delta)/2}}
\end{align*}
Recall that $\Q = N^\alpha$. Then $N=T^D$ in the general case of $|\mathcal{A}_\q^g|$. Thus we have $\Q _0 = T^{D\alpha_0}$, $Y=T^y$, $X=T^x$. For all of the above to hold, we need
\begin{align}
D\alpha_0 < 2(\delta-\theta)x \\
D\alpha_0 > 0 \\
x(1-\delta)< \alpha_0 \\
x(1-\delta)<y(\delta -\theta) \\
8d\alpha + 4y < (3+2\delta)x
\end{align}
To choose $\alpha_0$ we require $d(1-\delta)< 2(\delta- \theta)$, which is satisfied for say $\delta>.95$. To finish the proof of Theorem 8, observe that if we take $d=2$ and $\delta$ and $x$ very close to 1 equation (5) is satisfied and the last equation reduces to $\alpha< \frac{(3+2\delta)x - 4y}{16}$, which we can make larger than $\frac{5}{16} - \epsilon$ for any $\epsilon$ with appropriate choices of $x, \delta$. 
\\ \indent Now that we have the internal tools we'll need to complete Theorem 4, we introduce the remaining external tools we'll need. In our case, the sieve dimension $\kappa$ is the number of irreducible polynomials comprising a polynomial, and representing the easiest lower bound on the saturation number. The case with the best results is $\kappa=1$ corresponding to $z= c^2 + d^2$, for this we use the linear sieve developed by Greaves in \cite{G86}. For higher dimensional sieves we use the Diamond-Halberstam-Richert sieve developed in \cite{DHR96} and \cite{DH97}. The following computations of $\delta_0$ are the final ingredients to complete the proofs. 
\\ \indent Continuing the game we played above, supposing for now that $\delta>.95$, we are reduced to finding satisfactory values of $x, \delta,$ and  $\alpha$ for the last two equations. Recall that for $\delta>5/6$, we take $\theta=5/6$. Thus our two equations are equivalent to
\begin{align*}
x(1-\delta) < (1-x)(\delta-5/6) \Leftrightarrow x<6\delta-5 \\
8D\alpha + 4(1-x) < (3+2\delta)x \Leftrightarrow \frac{8D\alpha + 4}{7+2\delta} < x \\
\end{align*}
Treating this as a quadratic in $\delta$, we see that we will be able to solve the system as long as $(7+2\delta)(6\delta-5)>8D\alpha+4$ for some $5/6< \delta<1$. Expanding we obtain the quadratic $12\delta^2+32\delta - 8D\alpha - 39$ which has solutions $\frac{-32 \pm \sqrt{32^2 + 48(8D\alpha+39)}}{24}$. Denote the positive solution by $\delta_0$.
\\ \textbf{Hypotenuses:} From Greaves, as long as $\alpha>\frac{1}{4-.103974}$ we can attain $4$-almost primes. This gives us $\delta_0 \sim .983994188$. 
\\ \indent For the higher dimensional sieves, we use the Diamond-Halberstam-Richert sieve from \cite{DH97} and \cite{DHR96}. For the weights $\beta_\kappa$ computed there, using the simplification found in \cite{LS10} we can obtain any $R$ greater than 
\begin{align*}
\inf_{0<\zeta<\beta_k}m_{\alpha, \kappa}(\zeta) = \inf_{0<\zeta<\beta_k}\frac{1}{\alpha}(1 + \zeta - \frac{\zeta}{\beta_\kappa}) -1 + (\kappa + \zeta)\log(\frac{\beta_\kappa}{\zeta}) - \kappa  + \zeta\frac{\kappa}{\beta_\kappa} \\
\end{align*}
We note that $\beta_4 = 9.0722...$, $\beta_5= 11.5347...$ . 
\\ \textbf{Areas:}
We have $\kappa=4$, $D=4$, and optimizing $m_{\alpha, \kappa}(\zeta)$ with $\alpha=5/32$ gives $R>17.5...$. Seeing that the best $R$ we can attain is $18$, we seek to minimize $\alpha$ respecting this $R$. Using mathematica, we see that this is $\alpha \sim .1483334$. Now we compute $\delta_0 \sim .9954718$. Note that using $\alpha = 7/48$, the value implied from \cite{BK15}, the best R we can produce is 19.
\\ \textbf{Product of Coordinates:} We have $\kappa=5$, $D=6$, and optimizing $m_{\alpha, \kappa}(\zeta)$  with $\alpha = 5/48$ gives $R> 25.4...$, showing that our optimal R value is $26$. Again, we note that using $\alpha=7/72$, the value derived from \cite{BK15}, the best R we can produce is 27. Now again we compute that as long as $\alpha> .09980986..$, we can obtain $R=26$. Using this, our critical $\delta_0 \sim .99626261$.
\bibliographystyle{plain}
\bibliography{mookin}
\end{document}